\documentclass[12pt,leqno,draft]{article} 
\usepackage{amsmath,amssymb,amsthm,amsfonts,bm}
\usepackage{enumerate,color}
\makeatletter
\def\@cite#1#2{[{{\bfseries #1}\if@tempswa , #2\fi}]}
\renewcommand{\section}{%
\@startsection{section}{1}{\z@}
{0.5truecm plus -1ex minus -.2ex}%
{1.0ex plus .2ex}{\bfseries\large}}
\def\@seccntformat#1{\csname the#1\endcsname.\ }
\makeatother

\numberwithin{equation}{section} 
\newtheorem{thm}{Theorem}[section]

\newtheorem{lem}[thm]{Lemma}

\theoremstyle{definition}

\newtheorem*{prth1.1}{Proof of Theorem 1.1}
\newtheorem*{prth1.2}{Proof of Theorem 1.2}
\newtheorem*{prth1.3}{Proof of Theorem 1.3}

\newcommand{\ep}{\varepsilon}
\newcommand{\pa}{\partial}

\newcommand{\tmax}{T_{\rm max}}
\newcommand{\lp}[2]{\|#2\|_{L^{#1}(\Omega)}}

\begin{document}

\footnote[0]
    {2010{\it Mathematics Subject Classification}\/. 
    Primary: 35K45; Secondary: 92C17; 35Q35.
    }
\begin{center}
    \Large{{\bf Global existence and asymptotic behavior\\ 
    of classical solutions for a 
    3D two-species\\ chemotaxis-Stokes 
    system\\ with competitive kinetics 
               }}
\end{center}
\vspace{5pt}
\begin{center}
    Xinru Cao \\
    \vspace{2pt}
  Institut f\"ur Mathematik, Universit\"at Paderborn\\ 
  Warburger Str.100, 33098 Paderborn, Germany\\
    \vspace{12pt}
    Shunsuke Kurima\\
    \vspace{2pt}
    Department of Mathematics, 
    Tokyo University of Science\\
    1-3, Kagurazaka, Shinjuku-ku, Tokyo 162-8601, Japan\\
    \vspace{12pt}
     Masaaki Mizukami\footnote{Corresponding author}\\
    \vspace{2pt}
    Department of Mathematics, 
    Tokyo University of Science\\
    1-3, Kagurazaka, Shinjuku-ku, Tokyo 162-8601, Japan\\
    {\tt masaaki.mizukami.math@gmail.com}\\
\end{center}
\begin{center}    
    \small \today
\end{center}

\vspace{2pt}
\newenvironment{summary}
{\vspace{.5\baselineskip}\begin{list}{}{%
     \setlength{\baselineskip}{0.85\baselineskip}
     \setlength{\topsep}{0pt}
     \setlength{\leftmargin}{12mm}
     \setlength{\rightmargin}{12mm}
     \setlength{\listparindent}{0mm}
     \setlength{\itemindent}{\listparindent}
     \setlength{\parsep}{0pt} 
     \item\relax}}{\end{list}\vspace{.5\baselineskip}}
\begin{summary}
{\footnotesize {\bf Abstract.}
This paper considers the two-species chemotaxis-Stokes system 
with competitive kinetics 
\begin{equation*}
     \begin{cases}
         (n_1)_t + u\cdot\nabla n_1 =
          \Delta n_1 - \chi_1\nabla\cdot(n_1\nabla c) 
          + \mu_1n_1(1- n_1 - a_1n_2),
         &x\in \Omega,\ t>0,
 \\[2mm]
         (n_2)_t + u\cdot\nabla n_2 =
          \Delta n_2 - \chi_2\nabla\cdot(n_2\nabla c) 
          + \mu_2n_2(1- a_2n_1 - n_2),
         &x\in \Omega,\ t>0,
 \\[2mm]
         c_t + u\cdot\nabla c = 
         \Delta c - (\alpha n_1 +\beta n_2)c,
          &x \in \Omega,\ t>0,
 \\[2mm]
        u_t  
          = \Delta u + \nabla P 
            + (\gamma n_1 + \delta n_2)\nabla\phi, 
            \quad \nabla\cdot u = 0, 
          &x \in \Omega,\ t>0
     \end{cases} 
 \end{equation*} 
under homogeneous Neumann boundary conditions in a three-dimensional bounded domain 
$\Omega \subset \mathbb{R}^3$ with smooth boundary. 
Both chemotaxis-fluid systems and two-species 
chemotaxis systems with competitive terms 
are studied by many mathematicians. 
However, 
there has not been rich results 
on coupled two-species-fluid systems. 
Recently, global existence and asymptotic stability 
in the above problem with $(u\cdot \nabla)u$ in the 
fluid equation of the above system 
were established in the 2-dimensional case (\cite{HKMY_1}). 
The purpose of this paper is to give results 
for global existence, boundedness and stabilization 
of solutions to the above system 
in the 3-dimensional case 
when $\frac{\mu_i}{\chi_i}$ $(i=1,2)$ is sufficiently large.
}
\end{summary}
\vspace{10pt}

\newpage

\section{Introduction and results}

%
%
We consider the following 
two-species chemotaxis-fluid system 
with competitive terms:
 \begin{equation}\label{P}
     \begin{cases}
         (n_1)_t + u\cdot\nabla n_1 =
          \Delta n_1 - \chi_1\nabla\cdot(n_1\nabla c) 
          + \mu_1n_1(1- n_1 - a_1n_2),
         &x\in \Omega,\ t>0,
 \\[2mm]
         (n_2)_t + u\cdot\nabla n_2 =
          \Delta n_2 - \chi_2\nabla\cdot(n_2\nabla c) 
          + \mu_2n_2(1- a_2n_1 - n_2),
         &x\in \Omega,\ t>0,
 \\[2mm]
         c_t + u\cdot\nabla c = \Delta c -(\alpha n_1 + \beta n_2)c,
          &x \in \Omega,\ t>0,
 \\[2mm]
        u_t  + \kappa (u\cdot\nabla) u 
          = \Delta u + \nabla P 
            + (\gamma n_1 + \delta n_2)\nabla\phi, 
            \quad \nabla\cdot u = 0, 
          &x \in \Omega,\ t>0, 
 \\[2mm]
        \partial_\nu n_1 
        = \partial_\nu n_2 = \partial_\nu c = 0, \quad 
        u = 0, 
        &x \in \partial\Omega,\ t>0, 
 \\[2mm]
        n_i(x,0)=n_{i,0}(x),\ 
        c(x,0)=c_0(x),\ u(x,0)=u_0(x), 
        &x \in \Omega,\ i=1,2,
     \end{cases}
 \end{equation}
\noindent
where $\Omega$ is a bounded domain 
in $\mathbb{R}^3$ with smooth boundary $\partial\Omega$ and 
$\partial_\nu$ denotes differentiation with respect to the 
outward normal of $\partial\Omega$; 
$\kappa\in\{0,1\}$ 
(in this paper we will deal with 
the case that $\kappa=0$), 
$\chi_1, \chi_2, a_1, a_2 \ge 0$ and 
$\mu_1, \mu_2, \alpha, \beta, \gamma, \delta > 0$ are 
constants;  
$n_{1,0}, n_{2,0}, c_0, u_0, \phi$ 
are known functions satisfying
 \begin{align}\label{condi;ini1}
   &0 < n_{1,0}, n_{2,0} 
   \in C(\overline{\Omega}), 
 \quad 
   0 < c_0 \in W^{1,q}(\Omega), 
 \quad 
   u_0 \in D(A^{\vartheta}), 
  \\\label{condi;ini2}
   &\phi \in C^{1+\eta}(\overline{\Omega}) 
 \end{align}
for some $q > 3$, 
$\vartheta \in \left(\frac{3}{4}, 1\right)$, 
$\eta > 0$ and $A$ is 
the Stokes operator. 

%
%
The problem \eqref{P} is a generalized system 
to the chemotaxis-fluid system which is 
proposed by Tuval et al.\ \cite{Tuval_et_al}. 
This system
describes 
the evolution of 
two competing species 
which react on a single chemoattractant 
in a liquid surrounding environment. 
Here $n_1,n_2$ represent the 
population densities of species, 
$c$ stands for the concentration of chemoattractant, 
$u$ shows the fluid velocity field and 
$P$ represents the pressure of the fluid. 
The problem \eqref{P} comes from a problem 
on account of the influence of 
chemotaxis, 
the Lotka--Volterra competitive kinetics and the fluid. 
In the mathematical point of view, 
the chemotaxis term: $\nabla \cdot (n_1 \nabla c)$, 
the competition term: $n_1 (1-n_1-a_1 n_2)$ and 
the Stokes equation give difficulties in mathematical analysis. 

The one-species system \eqref{P} with $n_2=0$ has 
been studied in some literature. 
It is known that there exist global classical solutions 
in the 2-dimensional setting; however, 
in the 3-dimensional setting, 
only global weak solutions exist. 
In this one-species system with $\mu_1=0$, 
Winkler first attained global existence of 
classical solutions to \eqref{P}, $\kappa=0$
in the 3-dimensional setting 
and $\kappa=1$ 
in the 2-dimensional setting (\cite{W-2012}), 
and also established asymptotic stability 
of solutions to \eqref{P} (\cite{W-2014}). 
Moreover, the convergence rate has been already 
studied (\cite{Zhang-Li_2015_fluid}). 
Recently, 
Winkler \cite{Winkler_2017_Howfar} attained 
global existence 
and eventual smoothness of weak solutions 
and their 
asymptotic behavior 
for the 3-dimensional chemotaxis-Navier--Stokes 
system. 

In the analysis of the one-species case 
the logistic source can enhance the possibility of global existence of solutions. 
In the 3-dimensional setting, 
Lankeit \cite{Lankeit_2016} obtained 
global existence of weak solutions 
in \eqref{P} with $n_2=0$, 
$\kappa=1$ and with additional external force $f$ 
in the fourth equation, 
and also derived eventual smoothness 
and asymptotic behavior. 
Even for more complicated problems, 
Keller--Segel-fluid systems 
where 
 $-(\alpha n_1 + \beta n_2)c$ is 
replaced with 
$-c + \alpha n_1$ 
in \eqref{P} with $n_2=0$, 
logistic source is shown to be helpful for 
establishing classical bounded solutions. 
In the 3-dimensional setting, 
Tao and Winkler \cite{Tao-Winkler_2015_mu23} 
established 
global existence and boundedness of 
classical solutions by assuming that
$\mu_1 > 23$. 
In the 2-dimensional case, 
Tao and Winkler \cite{TW-2016} also showed 
global existence of bounded classical solutions 
in the Keller--Segel-Navier--Stokes system 
with logistic source 
with $+r n_1 -\mu_1 n_1^2$
 for any $\mu_1>0$, 
and their asymptotic behavior were obtained when $r=0$. 
For more related works 
we refer to 
Ishida \cite{Ishida_2015}, 
Wang and the first author \cite{Wang-Cao}, 
Wang and Xiang \cite{Wang-Xiang_2015}, 
Black \cite{Tobaias_2016}, 
the first author \cite{Xinru_2016}, 
the first author and Lankeit \cite{Xinru-Johannes_2016}, 
Kozono, Miura and Sugiyama \cite{Kozono-Miura-Sugiyama_2016}. 
These results fully parallel to those 
for the fluid free model; we can find 
counterpart in \cite{Johannes-Yulan,OTYM,Tao-Winkler_2012_consumption}.

On the other hand, the 
study on two-species competitive 
chemotaxis systems with signal consumption seems pending. 
We can only find related research with 
signal production in which 
the asymptotic behavior of solutions 
usually 
relies on some 
smallness assumption 
for the chemotaxis sensitivities 
(e.g., for the noncompetitive case 
($a_1=a_2=0$), see 
Negreanu and Tello \cite{N-T_SIAM,N-T_JDE}, 
the third author and Yokota \cite{MY-2016}, 
the third author \cite{M-2016}; 
for the competitive case see 
Tello and Winkler \cite{Tello_Winkler_2012},  
Stinner, Tello and Winkler \cite{stinner_tello_winkler}, 
Bai and Winkler \cite{B-W}, 
Black, Lankeit and the third author \cite{Black-Lankeit-Mizukami}, 
the third author \cite{Mizukami}). 

As mentioned above, the  
chemotaxis-fluid systems ($n_2=0$ in \eqref{P}) and 
the chemotaxis systems with competitive terms ($u=0$ in \eqref{P}) 
were studied by many mathematicians. 
However, the problem \eqref{P}, 
which is the combination of chemotaxis-fluid systems 
and chemotaxis-competition systems, 
had not been studied. 
Recently, global existence, 
boundedness of classical solutions 
and their asymptotic behavior 
were showed only in the 2-dimensional setting 
(\cite{HKMY_1}). 

The purpose of the present article is to 
obtain 
global existence and boundedness 
of classical solutions, and 
their asymptotic stability 
in the 3-dimensional setting. 
The main results read as follows. 
The first theorem gives global existence and 
boundedness in \eqref{P}. 
In view of known results on 
logistic chemotaxis-systems, 
it is no wonder that an assumption 
on smallness of $\chi_1$ and $\chi_2$ related to 
$\mu_1$ and $\mu_2$ will be necessary 
in the considered 3-dimensional case. 
%

 \begin{thm}\label{maintheorem1}
 Let $\Omega \subset \mathbb{R}^3$ be a 
 bounded domain with smooth boundary and let $\kappa=0$, $\chi_1, 
 \chi_2, a_1, a_2 \ge 0$, 
$\mu_1, \mu_2, \alpha, \beta, \gamma, \delta > 0$. 
Suppose that \eqref{condi;ini1} and \eqref{condi;ini2} hold. 
Then there exists a constant $\xi_0 > 0$ such that 
whenever 
\[
\chi := \max\{\chi_1, \chi_2\} \quad \mbox{and}\quad 
\mu := \min\{\mu_1, \mu_2\}
\] 
satisfy $\frac{\chi}{\mu} < \xi_0$, 
the problem \eqref{P} possesses a classical solution 
$(n_1, n_2, c, u, P)$ such that 
\begin{align*}
&n_1, n_2 \in C(\overline{\Omega}\times[0, \infty)) 
\cap C^{2, 1}(\overline{\Omega}\times(0, \infty)), 
\\
&c \in C(\overline{\Omega}\times[0, \infty)) 
\cap C^{2, 1}(\overline{\Omega}\times(0, \infty)) 
\cap L^{\infty}_{{\rm loc}}([0, \infty); W^{1, q}(\Omega)), 
\\
&u \in C(\overline{\Omega}\times[0, \infty)) 
\cap C^{2, 1}(\overline{\Omega}\times(0, \infty))
\cap L^\infty_{\rm loc}([0,\infty);D(A^{\vartheta})), 
\\  
&P \in C^{1, 0}(\overline{\Omega} \times (0, \infty)).
\end{align*} 
Also, the solution is unique in the sense that 
it allows up to 
addition of spatially constants to the pressure $P$. 
Moreover, there exists a constant $C>0$ such that 
   $$
   \|n_1(\cdot, t)\|_{L^{\infty}(\Omega)} + \|n_2(\cdot, t)\|_{L^{\infty}(\Omega)} 
   + \|c(\cdot, t)\|_{W^{1, q}(\Omega)} + \|u(\cdot, t)\|_{L^{\infty}(\Omega)} 
   \leq C 
   \quad \mbox{for all}\ t\in (0,\infty).
   $$
\end{thm}


The second theorem is concerned with 
asymptotic stability in \eqref{P}. 


 \begin{thm}\label{maintheorem3}
Let the assumption of Theorem \ref{maintheorem1} holds.  
Then the solution of \eqref{P} 
has the following properties\/{\rm :} 
   \begin{enumerate}
   \item[{{\rm (i)}}] Assume that $a_1, a_2 \in (0, 1)$. 
       Then 
       $$
       n_1(\cdot, t) \to N_1,\quad n_2(\cdot, t) \to N_2,\quad c(\cdot, t) \to 0,\quad 
       u(\cdot, t) \to 0 \quad \mbox{in}\ L^{\infty}(\Omega)\quad \mbox{as}\ t \to \infty, 
       $$
where 
       $$
       N_1 := \frac{1 - a_1}{1 - a_1a_2},\quad N_2 := \frac{1 - a_2}{1 - a_1a_2}.
       $$
   \item[{{\rm (ii)}}] Assume that $a_1 \geq 1 > a_2$. 
      Then 
      $$
      n_1(\cdot, t) \to 0,\quad n_2(\cdot, t) \to 1,\quad c(\cdot, t) \to 0,\quad 
       u(\cdot, t) \to 0 \quad \mbox{in}\ L^{\infty}(\Omega)\quad \mbox{as}\ t \to \infty.
      $$
   \end{enumerate}
\end{thm}

The strategy for the proof of Theorem \ref{maintheorem1} 
is to derive the $L^p$-estimate for $n_i$ with $p>\frac{3}{2}$. 
By using the differential inequality we can see 
\begin{align*}
 \int_\Omega n_1^p +\int_\Omega n_2^p 
 \le C 
 \int_{s_0}^t e^{-(p+1)(t-s)}\int_\Omega |\Delta c|^{p+1}  
 -C\int_{s_0}^{t}e^{-(p+1)(t-s)}
 \Bigl(\int_{\Omega}n_1^{p+1} + \int_{\Omega}n_2^{p+1} \Bigr)
\end{align*}
with some $C>0$ and $s_0>0$. 
The maximal Sobolev regularity 
(see Lemma \ref{pre2}) will be used to control 
$\int_{s_0}^t e^{-(p+1)(t-s)}\int_\Omega |\Delta c|^{p+1}$. 
Combining the maximal Sobolev regularity 
with some estimate for $|Au|^2$, we can obtain 
the $L^p$-estimate for $n_i$. 
On the other hand, 
the strategy for the proof of 
Theorem \ref{maintheorem3} is 
to derive the following inequality: 
\begin{align}\label{st;ineq}
  \int_0^\infty\int_\Omega (n_1-N_1)^2 
  + \int_0^\infty\int_\Omega (n_2-N_2)^2
  \le C
\end{align}
with some $C>0$, where $(N_1,N_2,0,0)$ 
is a constant solution to \eqref{P}. 
In order to obtain this estimate 
we will use the energy function 
\begin{align*}
E:=\int_\Omega 
       \left(
         n_1-N_1-N_1\log \frac{n_1}{N_1}
       \right)
       + 
       b_1\int_\Omega 
       \left(
         n_2-N_2-N_2\log \frac{n_2}{N_2}
       \right)
       + \frac{b_2}{2}
       \int_\Omega c^2 
\end{align*}
with some $b_1,b_2>0$, and show 
\begin{align*}
 \frac{d}{dt}E(t)\le 
 -\ep \int_\Omega \left[(n_1-N_1)^2+(n_2-N_2)^2\right]
\end{align*} 
with some $\ep>0$. 
This estimate and the positivity of $E(t)$ lead to 
\eqref{st;ineq}. 

This paper is organized as follows. 
In Section 2 we collect basic facts 
which will be used later. 
In Section 3 we prove global existence and boundedness (Theorem \ref{maintheorem1}). 
Sections 4 is devoted to showing 
asymptotic stability (Theorem \ref{maintheorem3}). 


\section{Preliminaries}

In this section we will provide 
some results which 
will be used later. 
The following lemma gives 
local existence of 
solutions to \eqref{P}.


 \begin{lem}\label{pre1}
 Let $\Omega \subset \mathbb{R}^3$ be a bounded domain with smooth boundary.  
Suppose that 
\eqref{condi;ini1} and \eqref{condi;ini2} hold. 
Then there exists $\tmax \in (0, \infty]$ 
such that the problem \eqref{P} possesses a classical solution 
$(n_1, n_2, c, u, P)$ fulfilling 
   \begin{align*}
&n_1, n_2 \in C(\overline{\Omega}\times[0, \tmax)) 
\cap C^{2, 1}(\overline{\Omega}\times(0, \tmax)), \\ 
&c \in C(\overline{\Omega}\times[0, \tmax)) 
\cap C^{2, 1}(\overline{\Omega}\times(0, \tmax)) 
\cap L^{\infty}_{{\rm loc}}([0, \tmax); W^{1, q}(\Omega)), \\ 
&u \in C(\overline{\Omega}\times[0, \tmax)) 
\cap C^{2, 1}(\overline{\Omega}\times(0, \tmax)) 
\cap L^\infty_{{\rm loc}}([0, \tmax); D(A^{\vartheta})), 
\\
&n_1,n_2>0,\quad c>0\quad \mbox{in}\ 
\Omega\times (0,\tmax).
  \end{align*}
Also, the solution is unique up to addition of spatially constants to the pressure $P$. 
Moreover, either $\tmax = \infty$ or 
   $$
   \limsup_{t\nearrow \tmax}
    (\|n_1(\cdot, t)\|_{L^{\infty}(\Omega)} + \|n_2(\cdot, t)\|_{L^{\infty}(\Omega)} 
    + \|c(\cdot, t)\|_{W^{1, q}(\Omega)} 
    + \|A^{\vartheta}u(\cdot, t)\|_{L^2(\Omega)}) 
   = \infty.
   $$         
 \end{lem}
\begin{proof}
The proof of local existence of 
classical solutions to \eqref{P} 
is based on a standard contraction mapping argument, 
which can be found in 
\cite{W-2012}. 
Accordingly, the maximum principle is 
applied to yield $n_1,n_2 > 0$ and $c > 0$
in $\Omega\times(0, \tmax)$. 
\end{proof}

Given all $s_0 \in (0, \tmax)$, 
from the regularity properties we see that 
  \[c(\cdot, s_0) \in C^2(\overline{\Omega})
\quad 
\mbox{with}\quad 
\partial_{\nu}c(\cdot, s_0) = 0 
\quad \mbox{on}\ \partial\Omega. 
\]
In particular, there exists 
$M = M(s_0)> 0$ such that 
   $$ 
   \|c(\cdot, s_0)\|_{W^{2, \infty}(\Omega)} \leq M
   $$
(see e.g., \cite{YCJZ-2015}). 

The following lemma is referred to as a variation of 
the maximal Sobolev regularity 
(see \cite[Theorem 3.1]{Hieber-Press_Maximal-Sobolev}), 
which is important to prove Theorem \ref{maintheorem1}. 

\begin{lem}\label{pre2} 
Let $s_0 \in (0, \tmax)$. 
Then for all $p>1$ 
there exists a constant $C_1=C_1(p)>0$ 
such that 
\begin{align*}
  \int_{s_0}^{t}\int_{\Omega}e^{ps}|\Delta c|^p
  &\leq C_1\int_{s_0}^{t}\int_{\Omega}e^{ps}
       |2c-(\alpha n_1+\beta n_2)c - 
       u\cdot \nabla c|^p 
\\&\quad\,
       + C_1 e^{ps_0}
       (\|c(\cdot, s_0)\|_{L^p(\Omega)}^p 
       + \|\Delta c(\cdot, s_0)\|_{L^p(\Omega)}^p)
\end{align*}
holds for all $t\in (s_0,\tmax)$. 
\end{lem}
\begin{proof}
 Let $s_0\in (0,\tmax)$ and let $t\in (s_0,\tmax)$. 
 We rewrite the third equation as
 \[
 c_t = (\Delta -1) c - c + \bigl(2c -(\alpha n_1+\beta n_2)c\bigr) 
 - u\cdot \nabla c, 
 \] 
and use the transformation 
$\widetilde{c}(\cdot, s)=e^s c(\cdot, s)$, $s\in (s_0,t)$. 
Then $\widetilde{c}$ satisfies 
\begin{align*}
  \begin{cases}
   \widetilde{c}_t = (\Delta -1) \widetilde{c} + f, 
   & x\in\Omega,\ s\in(s_0,t), 
  \\
   \partial_\nu \widetilde{c} = 0,
   & x\in\pa\Omega,\ s\in (s_0,t), 
  \\ 
   \widetilde{c}(\cdot,s_0)=e^{s_0}c(\cdot,s_0)
   \in W^{2,p}(\Omega), 
   & x\in\Omega, 
   \end{cases}
\end{align*} 
where 
\begin{align*}
 f := e^s(2c -(\alpha n_1 + \beta n_2)c 
          - u\cdot \nabla c) \in L^p(s_0,t;L^p(\Omega)). 
\end{align*} 
Therefore an application of the maximal Sobolev regularity 
\cite[Theorem 3.1]{Hieber-Press_Maximal-Sobolev} to 
$\widetilde{c}$ implies this lemma. 
\end{proof}

\vspace{10pt}

\section{Boundedness. Proof of Theorem \ref{maintheorem1}}

In this section we will prove 
Theorem \ref{maintheorem1} 
by preparing a series of lemmas.


\begin{lem}\label{pote1}
There exists a constant $C_2 > 0$ such that 
     $$
     \int_{\Omega} n_i(\cdot, t) \leq C_2
     $$
for all $t \in (0, \tmax)$ for $i = 1,2$.
\end{lem}
\begin{proof}
The same argument as in the proof of \cite[Lemma 3.1]{HKMY_1} 
implies this lemma.
\end{proof}
 \begin{lem}\label{lem;Linf;c}
The function $t \mapsto \lp{\infty}{c(\cdot, t)}$ is nonincreasing. 
In particular, 
    \begin{equation*}
    \|c(\cdot, t)\|_{L^{\infty}(\Omega)} \le \|c_0\|_{L^{\infty}(\Omega)} 
    \end{equation*}
holds for all $t \in (0, \tmax)$. 
 \end{lem}
\begin{proof}
We can prove this lemma by applying the maximum principle 
to the third equation in \eqref{P}.
\end{proof}


\begin{lem}\label{pote3}
For $r \in (1, 3)$ there exists a constant 
$C_3=C_3(r) > 0$ such that 
     $$
     \|u(\cdot, t)\|_{L^r(\Omega)} \leq C_3
     $$
for all $t \in (0, \tmax)$.
\end{lem}
\begin{proof}
From the well-known 
Neumann heat semigroup estimates together with 
Lemma \ref{pote1} we can obtain 
the $L^r$-estimate for $u$ with $r\in (1,3)$ 
(for more details, see \cite[Corollary 3.4]{W-2015}). 
\end{proof}

Now we fix $s_0 \in (0, \tmax) \cap (0, 1]$.
The proofs of the following 
two lemmas are based on the methods 
in \cite[Lemma 3.1]{YCJZ-2015}. 

\begin{lem}\label{pote4}
For all $p > 1$, $\ep > 0$ and $\ell > 0$ 
there exists a constant $C_4=C_4(p) > 0$ 
such that
 \begin{align*}
     &\frac{1}{p}\int_{\Omega}n_1^p(\cdot, t) + \frac{1}{p}\int_{\Omega}n_2^p(\cdot, t) \\ 
     &\leq - (\mu - \ep - \ell)e^{-(p+1)t}\int_{s_0}^{t}e^{(p+1)s}
                                 \Bigl(\int_{\Omega}n_1^{p+1} + \int_{\Omega}n_2^{p+1} \Bigr) \\ 
              &\quad\,+ C_4\ell^{-p}\chi^{p+1}e^{-(p+1)t}
                               \int_{s_0}^{t}e^{(p+1)s}\int_{\Omega}|\Delta c|^{p+1} 
                      + C_4
     \end{align*}
for all $t \in (s_0, \tmax)$, 
where $\mu = \min\{\mu_1,\mu_2\}$, 
$\chi = \max\{\chi_1, \chi_2\}$.
\end{lem}
\begin{proof}
Let $p > 1$.  
Multiplying the first equation in { \eqref{P}} by $n_1^{p-1}$ 
and integrating it over $\Omega$, we see that 
   \begin{align*}
   \frac{1}{p}\frac{d}{dt}\int_{\Omega}n_1^p  
   &= -\frac{1}{p}\int_{\Omega} u \cdot \nabla n_1^p 
        + \int_{\Omega}n_1^{p-1}\Delta n_1 
        + \chi_1\frac{p-1}{p}\int_{\Omega}\nabla n_1^p\cdot\nabla c \\ \notag
        &\quad\,+ \mu_1\int_{\Omega}n_1^p - \mu_1\int_{\Omega}n_1^{p+1} 
        - a_1\mu_1\int_{\Omega}n_1^p n_2.  
   \end{align*} 
Noting from $\nabla \cdot u = 0$ in $\Omega\times(0, \tmax)$ that 
$\int_{\Omega} u \cdot \nabla n_1^p = -\int_{\Omega} (\nabla \cdot u)n_1^p = 0$, 
we obtain from integration by parts and nonnegativity of $n_1$, $n_2$ that 
     \begin{align}\label{pote1.1}
   \frac{1}{p}\frac{d}{dt}\int_{\Omega}n_1^p  
   &= - (p-1)\int_{\Omega}n_1^{p-2}|\nabla n_1|^2 
        -\chi_1\frac{p-1}{p}\int_{\Omega}n_1^p\Delta c 
            + \mu_1\int_{\Omega}n_1^p - \mu_1\int_{\Omega}n_1^{p+1} 
         \\ \notag
              &\quad\,- a_1\mu_1\int_{\Omega}n_1^p n_2
   \\ \notag 
   &\leq -\chi_1\frac{p-1}{p}\int_{\Omega}n_1^p\Delta c 
            + \mu_1\int_{\Omega}n_1^p - \mu_1\int_{\Omega}n_1^{p+1} 
   \\ \notag
   &= -\frac{p+1}{p}\int_{\Omega}n_1^p 
        -\chi_1\frac{p-1}{p}\int_{\Omega}n_1^p\Delta c 
        + \Bigl(\mu_1 + \frac{p+1}{p} \Bigr)\int_{\Omega}n_1^p 
        - \mu_1\int_{\Omega}n_1^{p+1}.
       \end{align}%
Now we let $\ep > 0$ and $\ell >0$. 
By the Young inequality there 
exists a constant $C_5 = C_5(\mu_1, \ep, p) > 0$ such that 
      \begin{equation}\label{pote1.2}
      \Bigl(\mu_1 + \frac{p+1}{p} \Bigr)\int_{\Omega}n_1^p 
      \leq \ep\int_{\Omega}n_1^{p+1} + C_5.
      \end{equation}
Moreover, the second term on the right-hand side of 
\eqref{pote1.1} can be estimated as 
     \begin{equation}\label{pote1.3}
     -\chi_1\frac{p-1}{p}\int_{\Omega}n_1^p\Delta c 
     \leq \chi_1\int_{\Omega}n_1^p|\Delta c| 
     \leq \ell\int_{\Omega}n_1^{p+1} + C_6\ell^{-p}\chi_1^{p+1}\int_{\Omega}|\Delta c|^{p+1} 
     \end{equation}
with some $C_6 = C_6(p) > 0$. 
Hence we derive from \eqref{pote1.1}, \eqref{pote1.2} and \eqref{pote1.3} that 
      \begin{align*}
      \frac{1}{p}\frac{d}{dt}\int_{\Omega}n_1^p + \frac{p+1}{p}\int_{\Omega}n_1^p
      &\leq - (\mu_1 - \ep - \ell)
      \int_{\Omega}n_1^{p+1} \\ \notag
      &\quad\,+ C_6\ell^{-p}\chi_1^{p+1}\int_{\Omega}|\Delta c|^{p+1} 
      + C_5. 
      \end{align*} 
Therefore there exists 
$C_7 = C_7(\mu_1, \ep, p, |\Omega|, s_0) > 0$ such that  
     \begin{align}\label{pote1.5}
     \frac{1}{p}\int_{\Omega}n_1^p(\cdot, t) 
     &\leq e^{-(p+1)(t-s_0)}\frac{1}{p}\int_{\Omega}n_1^p(\cdot, s_0) 
     - (\mu_1 - \ep - \ell)\int_{s_0}^{t}
     e^{-(p+1)(t-s)}\int_{\Omega}n_1^{p+1}
     \\ \notag
        &\quad\,+ C_6\ell^{-p}\chi_1^{p+1}
                               \int_{s_0}^{t}e^{-(p+1)(t-s)}\int_{\Omega}|\Delta c|^{p+1} 
              + C_5\int_{s_0}^{t}e^{-(p+1)(t-s)} \\ \notag 
     &\leq - (\mu_1 - \ep - \ell)e^{-(p+1)t}\int_{s_0}^{t}e^{(p+1)s}\int_{\Omega}n_1^{p+1} \\ \notag
              &\quad\,+ C_6\ell^{-p}\chi_1^{p+1}e^{-(p+1)t}
                               \int_{s_0}^{t}e^{(p+1)s}\int_{\Omega}|\Delta c|^{p+1} 
                      + C_7
     \end{align}
for each $t \in (s_0, \tmax)$. 
Similarly, we see that 
     \begin{align}\label{pote1.6}
     \frac{1}{p}\int_{\Omega}n_2^p(\cdot, t) 
     &\leq - (\mu_2 - \ep - \ell)e^{-(p+1)t}\int_{s_0}^{t}e^{(p+1)s}\int_{\Omega}n_2^{p+1} \\ \notag
              &\quad\,+ C_6\ell^{-p}\chi_2^{p+1}e^{-(p+1)t}
                               \int_{s_0}^{t}e^{(p+1)s}\int_{\Omega}|\Delta c|^{p+1} 
                      + C_8
     \end{align}
with some $C_8 = C_8(\mu_2, \ep, p, |\Omega|, s_0) > 0$. 
Thus from \eqref{pote1.5} 
and \eqref{pote1.6} we have 
that there exists 
$C_9 = C_9(\mu, \ep, p, |\Omega|, s_0) > 0$ such that 
      \begin{align*}
     &\frac{1}{p}\int_{\Omega}n_1^p(\cdot, t) + \frac{1}{p}\int_{\Omega}n_2^p(\cdot, t) \\ \notag
     &\leq - (\mu - \ep - \ell)e^{-(p+1)t}\int_{s_0}^{t}e^{(p+1)s}
                                 \Bigl(\int_{\Omega}n_1^{p+1} + \int_{\Omega}n_2^{p+1} \Bigr) \\ \notag
              &\quad\,+ 2C_6\ell^{-p}\chi^{p+1}e^{-(p+1)t}
                               \int_{s_0}^{t}e^{(p+1)s}\int_{\Omega}|\Delta c|^{p+1} 
                      + C_9, 
     \end{align*}
where $\mu = \min\{\mu_1, \mu_2\}$ 
and 
$\chi = \max\{\chi_1, \chi_2\}$.  
\end{proof}

\begin{lem}\label{pote5}
For all $p \in (1, 2)$ there exists a constant 
$C_{10}=C_{10}(p) > 0$ such that 
 \begin{align*}
     &\int_{s_0}^{t}e^{(p+1)s}\int_{\Omega}|\Delta c|^{p+1} \\ 
      &\leq C_{10}\int_{s_0}^{t}e^{(p+1)s}
     (\|n_1\|_{L^{p+1}(\Omega)}^{p+1} + 
     \|n_2\|_{L^{p+1}(\Omega)}^{p+1}) 
      + C_{10}\int_{s_0}^{t}
      e^{(p+1)s}
      \|Au\|_{L^2(\Omega)}^2\,{ ds} 
      \\ 
      &\quad\,+ C_{10}e^{(p+1)t} + C_{10}
     \end{align*}
for all $t \in (s_0, \tmax)$.
\end{lem}
\begin{proof}
Fix $\theta \in (1, 2)$ and put $\theta' = \frac{\theta}{\theta - 1}$.
We derive from Lemma \ref{pre2} that 
      \begin{align}\label{pote1.8}
      \int_{s_0}^{t}e^{(p+1)s}\int_{\Omega}|\Delta c|^{p+1} 
      &\leq C_{11}\int_{s_0}^{t}e^{(p+1)s}
      \int_{\Omega}
      |2c - (\alpha n_1 + \beta n_2)c - u\cdot\nabla c|^{p+1} 
      \\ \notag 
          &\quad\,+ C_{11}e^{(p+1)s_0}
          \|c(\cdot, s_0)\|_{W^{2, p+1}(\Omega)}^{p+1} 
      \end{align}
holds with some $C_{11}=C_{11}(p) > 0$. 
Lemma \ref{lem;Linf;c} and 
the H\"older inequality imply 
      \begin{align}\label{pote1.8tochu}
      &\int_{s_0}^{t}e^{(p+1)s}\int_{\Omega}
      |2c - (\alpha n_1 + \beta n_2)c - u\cdot\nabla c|^{p+1} 
      \\ \notag 
      &\leq C_{12}\int_{s_0}^{t}e^{(p+1)s}
      \int_{\Omega} 
          (n_1^{p+1} + n_2^{p+1} + |u\cdot\nabla c|^{p+1}) 
          + C_{12}e^{(p+1)t} 
      \\ \notag 
      &\leq C_{12}\int_{s_0}^{t}e^{(p+1)s}
     \Bigl(\|n_1\|_{L^{p+1}(\Omega)}^{p+1} + \|n_2\|_{L^{p+1}(\Omega)}^{p+1} 
               + \|u\|_{L^{(p+1)\theta}(\Omega)}^{p+1}
               \|\nabla c\|_{L^{(p+1)\theta'}(\Omega)}^{p+1}\Bigr) 
    \\ \notag
     &\quad\,+ C_{12}e^{(p+1)t}
      \end{align}
with some $C_{12}=C_{12}(p) > 0$. 
Here we see from the Gagliardo--Nirenberg 
inequality and Lemma \ref{lem;Linf;c} 
that there exist constants 
$C_{13}=C_{13}(p), C_{14}=C_{14}(p) > 0$ such that 
      \begin{align}\label{pote1.9}
      \|\nabla c\|_{L^{(p+1)\theta'}(\Omega)}^{p+1} 
      &\leq C_{13}\|\Delta c\|_{L^{p+1}(\Omega)}^{a(p+1)}\|c\|_{L^\infty(\Omega)}^{(1-a)(p+1)} 
             + C_{13}\|c\|_{L^1(\Omega)}^{p+1} \\ \notag 
      &\leq C_{14}\|\Delta c\|_{L^{p+1}(\Omega)}^{a(p+1)} 
      + C_{14}
      \end{align}
with $a := \frac{1-\frac{3}{(p+1)\theta'}}{2-\frac{3}{p+1}} 
\in (\frac{1}{2}, 1)$. 
By \eqref{pote1.8}, \eqref{pote1.8tochu}, \eqref{pote1.9} and the Young inequality 
it holds that 
     \begin{align*}
     \int_{s_0}^{t}e^{(p+1)s}
     \int_{\Omega}|\Delta c|^{p+1} 
      &\leq C_{15}\int_{s_0}^{t}e^{(p+1)s}
     (\|n_1\|_{L^{p+1}(\Omega)}^{p+1} + \|n_2\|_{L^{p+1}(\Omega)}^{p+1}) \\ \notag
           &\quad\,+ a\int_{s_0}^{t}e^{(p+1)s}\|\Delta c\|_{L^{p+1}(\Omega)}^{p+1} 
      + C_{15}\int_{s_0}^{t}e^{(p+1)s} 
     \|u\|_{L^{(p+1)\theta}(\Omega)}^{\frac{p+1}{1-a}} \\ \notag
        &\quad\,+ C_{15}\int_{s_0}^{t}e^{(p+1)s}\|u\|_{L^{(p+1)\theta}(\Omega)}^{p+1}   
          + C_{15}e^{(p+1)t}
          + C_{15}
     \end{align*} 
with some $C_{15}=C_{15}(p) > 0$. 
Here we use $p<2$, which namely enable 
us to pick $r \in (1, 3)$ such that 
     \begin{equation}\label{keypoint}
     \frac{2-\frac{3}{p+1}}{1-\frac{3}{(p+1)\theta}}\cdot(p+1)\cdot
                        \frac{\frac{3}{r} - \frac{3}{(p+1)\theta}}{\frac{1}{2} + \frac{3}{r}} 
     < 2
     \end{equation}
holds.
Therefore we can obtain that  
     \begin{align}\label{pote1.11}
     &\int_{s_0}^{t}e^{(p+1)s}\int_{\Omega}|\Delta c|^{p+1} \\ \notag
      &\leq \frac{C_{15}}{1-a}\int_{s_0}^{t}e^{(p+1)s}
     (\|n_1\|_{L^{p+1}(\Omega)}^{p+1} + 
     \|n_2\|_{L^{p+1}(\Omega)}^{p+1}) 
      + \frac{C_{15}}{1-a}\int_{s_0}^{t}e^{(p+1)s}
      \|u\|_{L^{(p+1)\theta}(\Omega)}^{\frac{p+1}{1-a}} \\ \notag
      &\quad\,+ \frac{C_{15}}{1-a}\int_{s_0}^{t}e^{(p+1)s}
      \|u\|_{L^{(p+1)\theta}(\Omega)}^{p+1} + \frac{C_{15}}{1-a}e^{(p+1)t}
      + \frac{C_{15}}{1-a}.
     \end{align} 
Now we see from the Gagliardo--Nirenberg inequality, Lemma \ref{pote3} and the Young inequality 
that there exists a constant $C_{16}=C_{16}(p)>0$ 
such that 
     \begin{equation}\label{pote1.12}
     \|u(\cdot, s)\|_{L^{(p+1)\theta}(\Omega)}^{\frac{p+1}{1-a}} 
     \leq \|Au(\cdot, s)\|_{L^2(\Omega)}^{\frac{p+1}{1-a}b}
     \|u(\cdot, s)\|_{L^r(\Omega)}^{\frac{p+1}{1-a}(1-b)} 
     \leq C_{16} + C_{16}\|Au(\cdot, s)\|_{L^2(\Omega)}^2 
     \end{equation} 
with $b := \frac{\frac{3}{r} - \frac{3}{(p+1)\theta}}{\frac{1}{2} + \frac{3}{r}} 
\in (0, 1)$, since $\frac{p+1}{1-a}b < 2$ 
from \eqref{keypoint}. 
Similarly, there exists a constant 
$C_{17}=C_{17}(p)>0$ 
such that 
     \begin{equation}\label{pote1.12.1}
     \|u(\cdot, s)\|_{L^{(p+1)\theta}(\Omega)}^{p+1} 
     \leq C_{17} + C_{17}\|Au(\cdot, s)\|_{L^2(\Omega)}^2
     \end{equation}
for all $s \in (s_0, \tmax)$. Therefore combination of \eqref{pote1.11} with \eqref{pote1.12} and \eqref{pote1.12.1} yields that 
there exists a constant $C_{18}=C_{18}(p) > 0$ 
such that 
\begin{align*}
     &\int_{s_0}^{t}e^{(p+1)s}\int_{\Omega}|\Delta c(\cdot, s)|^{p+1}\,ds \\ 
      &\leq C_{18}\int_{s_0}^{t}e^{(p+1)s}
     (\|n_1(\cdot, s)\|_{L^{p+1}(\Omega)}^{p+1} + \|n_2(\cdot, s)\|_{L^{p+1}(\Omega)}^{p+1})\,ds \\
&\quad\,
      + C_{18}\int_{s_0}^{t}e^{(p+1)s}\|Au(\cdot, s)\|_{L^2(\Omega)}^2\,ds 
        + C_{18}e^{(p+1)t} + C_{18}
     \end{align*}
for all $t \in (s_0, \tmax)$, which means the end of the proof. 
\end{proof}


In order to control $\int_{s_0}^{t}e^{(p+1)s}\int_{\Omega}|\Delta c|^{p+1}$ 
we provide the following lemma.
\begin{lem}\label{pote6}
For all $p > 1$ there exists a constant 
$C_{19}=C_{19}(p) > 0$ such that 
 \begin{align*}
     &\int_{s_0}^{t}e^{(p+1)s}\int_{\Omega}|Au|^2  \\ 
     &\leq { 2}e^{(p+1)s_0}\int_{\Omega}|A^{\frac{1}{2}}u(\cdot, s_0)|^2 
              + C_{19}e^{(p+1)t} 
            + C_{19}\int_{s_0}^{t}e^{(p+1)s}\Bigl(\int_{\Omega}n_1^{p+1} + \int_{\Omega}n_2^{p+1} \Bigr)
     \end{align*}
for all $t \in (s_0, \tmax)$.
\end{lem}
\begin{proof}
It follows from the fourth equation in \eqref{P}, 
the Young inequality 
and 
the continuity of the Helmholtz projection on 
$L^2(\Omega; \mathbb{R}^3)$ 
(\cite[Theorem 1]{FM}) that 
     \begin{align*}
     &\frac{1}{2}\frac{d}{dt}\int_{\Omega}|A^{\frac{1}{2}}u|^2 + \int_{\Omega}|Au|^2 \\
     &= \int_{\Omega} Au \cdot {\cal P}[(\gamma n_1 + \delta n_2)\nabla\phi] \\
     &\leq \frac{1}{4}\int_{\Omega} |Au|^2 
                            + \int_{\Omega}|(\gamma n_1 + \delta n_2)\nabla\phi|^2 \\
     &\leq \frac{1}{4}\int_{\Omega} |Au|^2 
 + (\gamma^2 + \delta^2)\|\nabla\phi\|_{L^{\infty}(\Omega)}^2\Bigl(\int_{\Omega}n_1^2 + \int_{\Omega}n_2^2 \Bigr) \\
     &\leq \frac{1}{4}\int_{\Omega} |Au|^2 + (\gamma^2 + \delta^2)\|\nabla\phi\|^2_{L^{\infty}(\Omega)}\Bigl(2|\Omega| + \int_{\Omega}n_1^{p+1} 
                                                               + \int_{\Omega}n_2^{p+1} \Bigr), 
     \end{align*} 
and hence there exists a constant $C_{20} > 0$ such that 
     \begin{equation}\label{pote1.13}
     \frac{1}{2}\frac{d}{dt}\int_{\Omega}|A^{\frac{1}{2}}u|^2 
     + \frac{3}{4}\int_{\Omega}|Au|^2 
     \leq C_{20} + C_{20}\Bigl(\int_{\Omega}n_1^{p+1} + \int_{\Omega}n_2^{p+1} \Bigr)
     \end{equation}
and we derive from \cite[Part2, Theorem 14.1]{Friedman_partial}, 
Lemma \ref{pote3} and the Young inequality that 
     \begin{equation}\label{pote1.14}
     \int_{\Omega}|A^{\frac{1}{2}}u|^2 
     \leq C_{21}\|Au(\cdot, s)\|_{L^2(\Omega)}\|u\|_{L^2(\Omega)}
     \leq C_{22} + \frac{1}{(p+1)}\|Au(\cdot, s)\|_{L^2(\Omega)}^2
     \end{equation}%
with some constants $C_{21}, C_{22}>0$. 
By \eqref{pote1.13} and \eqref{pote1.14} we obtain 
     \begin{equation*}
     \frac{1}{2}\frac{d}{dt}\int_{\Omega}|A^{\frac{1}{2}}u|^2 
     + \frac{p+1}{2}\int_{\Omega}|A^{\frac{1}{2}}u|^2 
     + \frac{1}{4}\int_{\Omega}|Au|^2 
     \leq C_{23} 
             + C_{23}\Bigl(\int_{\Omega}n_1^{p+1} + \int_{\Omega}n_2^{p+1} \Bigr) 
     \end{equation*} 
with some constant $C_{23}=C_{23}(p) > 0$, and hence we have 
     \begin{align*}
     &\int_{s_0}^{t}e^{(p+1)s}\int_{\Omega}|Au|^2  \\ \notag
     &\leq { 2}e^{(p+1)s_0}\int_{\Omega}|A^{\frac{1}{2}}u(\cdot, s_0)|^2 
            + C_{24}e^{(p+1)t} 
            + C_{24}\int_{s_0}^{t}e^{(p+1)s}
                         \Bigl(\int_{\Omega}n_1^{p+1} + \int_{\Omega}n_2^{p+1} \Bigr)
     \end{align*}
with some constant $C_{24}=C_{24}(p) > 0$, 
which concludes the proof.
\end{proof}

\begin{lem}\label{pote7}
For all $p \in (1, 2)$ and for all $\ell > 0$ 
there exist positive constants 
$K(p) > 0$ and 
$C_{25}=C_{25}(p, \ell)> 0$ such that 
if $\mu > \mu_{p, \ell} := \ell + K(p)\ell^{-p}\chi^{p+1}$, then 
     $$
     \|n_i(\cdot, t)\|_{L^p(\Omega)} \leq C_{25} 
     \quad \mbox{for all}\ t \in (s_0, \tmax)
     $$
for $i=1, 2$, 
where 
$\mu = \min\{\mu_1,\mu_2\}$, $\chi = \max\{\chi_1, \chi_2\}$.
\end{lem}
\begin{proof}
It follows from Lemmas \ref{pote4}, \ref{pote5} and \ref{pote6} that 
there exists a constant $K(p) > 0$ such that 
     \begin{align*}
     &\frac{1}{p}\int_{\Omega}n_1^p(\cdot, t) + \frac{1}{p}\int_{\Omega}n_2^p(\cdot, t) \\ \notag
     &\leq - (\mu - \ep - \ell - K(p)\ell^{-p}\chi^{p+1})e^{-(p+1)t}\int_{s_0}^{t}e^{(p+1)s}
                                 \Bigl(\int_{\Omega}n_1^{p+1} + \int_{\Omega}n_2^{p+1} \Bigr) \\ \notag
              &\quad\,+ K(p)\ell^{-p}\chi^{p+1} + K(p)
     \end{align*}
for all $t \in (s_0, \tmax)$. 
We assume that $\mu > \mu_{p, \ell}$. 
Then there exists $\ep \in (0, \mu - \mu_{p, \ell})$ such that 
         $$
         \mu - \ep - \ell - K(p)\ell^{-p}\chi^{p+1} \geq 0. 
         $$
Thus we derive that 
         $$
         \frac{1}{p}\int_{\Omega}n_1^p(\cdot, t) + \frac{1}{p}\int_{\Omega}n_2^p(\cdot, t) 
         \leq K(p)\ell^{-p}\chi^{p+1} + K(p)
         $$
holds for all $t \in (s_0, \tmax)$, 
which concludes the proof of Lemma \ref{pote7}.
\end{proof}

The proof of the following lemma is based on the method in 
\cite[Proof of Theorem 1]{YCJZ-2015}.
\begin{lem}\label{pote8}
For all $p \in (\frac{3}{2}, 2)$ there exists 
$\xi_0 > 0$ such that 
if $\frac{\chi}{\mu} < \xi_0$, then  
there exists a constant $C_{26} > 0$ such that 
     $$
     \|n_i(\cdot, t)\|_{L^p(\Omega)} \leq C_{26} 
     \quad \mbox{for all}\ t \in (0, \tmax)
     $$
for $i=1, 2$.
\end{lem}
\begin{proof} 
Let $p \in (\frac{3}{2}, 2)$ and 
let $\xi_0=\xi_0 (p) > 0$ satisfing 
       $$
       \inf_{\ell > 0} \mu_{p, \ell} = \inf_{\ell > 0} (\ell + K(p)\ell^{-p}\chi^{p+1}) 
       = \frac{1}{\xi_0}\chi.
       $$
Then we can see that $\frac{\chi}{\mu} < \xi_0$ implies $\mu > \mu_{p, \ell}$ with some  $\ell > 0$. 
Therefore Lemma \ref{pote7} implies that 
there exists a constant $C_{27}=C_{27}(p) > 0$ such that 
     $$
     \|n_i(\cdot, t)\|_{L^p(\Omega)} \leq C_{27}
     \quad \mbox{for all}\ t \in (0, \tmax)
     $$
for each $i=1, 2$, which implies the end of the proof. 
\end{proof}

\begin{lem}\label{pote9}
Assume $\frac{\chi}{\mu} < \xi_0$. Then there 
exists a constant $C_{28} > 0$ such that 
     \[
     \|A^{\vartheta}u(\cdot, t)\|_{L^2(\Omega)} \leq C_{28}
    \quad \mbox{and}\quad  
    \lp{\infty}{u(\cdot,t)}\le C_{28}
    \]
for all $t\in (0,\tmax)$. 
\end{lem}
\begin{proof}
Noting that $\frac{1}{2}+\frac{2}{3}(1-\vartheta) \in (\frac{1}{2}, \frac{2}{3})$,
we can fix 
$p \in \Bigl(\frac{1}{\frac{1}{2}+\frac{2}{3}(1-\vartheta)}, 2\Bigr)$. 
It follows from Lemma \ref{pote8}, the well-known
regularization estimates for Stokes semigroup \cite{Gigi, Sohr} 
and the continuity of the Helmholtz projection on 
$L^r(\Omega; \mathbb{R}^3)$ (see e.g.,  \cite[Theorem 1]{FM}) 
that there exist constants $C_{29}, C_{30}, C_{31}, C_{32} > 0$ 
and $\lambda > 0$ such that
      \begin{align*}
      \|A^{\vartheta}u(\cdot, t)\|_{L^2(\Omega)} 
      &\leq \|A^{\vartheta}e^{-tA}u_0\|_{L^2(\Omega)} 
               + \int_0^t 
           \|A^{\vartheta}e^{-(t-s)A}{\cal P}[(\gamma n_1 + \delta n_2)\nabla\phi]\|_{L^2(\Omega)}\,ds
 \\ 
       &
       \leq \|A^{\vartheta}u_0\|_{L^2(\Omega)}
 \\
    &\quad\, 
    + C_{29}\int_0^t 
    (t-s)^{-\vartheta-\frac{3}{2}(\frac{1}{p}-\frac{1}{2})}
    e^{-\lambda(t-s)}\|{\cal P}[(\gamma n_1 + \delta n_2)\nabla\phi]\|_{L^p(\Omega)}\, ds 
    \\
       &\leq \|A^{\vartheta}u_0\|_{L^2(\Omega)} 
                + C_{30}\int_0^t 
                    (t-s)^{-\vartheta-\frac{3}{2}(\frac{1}{p}-\frac{1}{2})}e^{-\lambda(t-s)}\|\gamma n_1+\delta n_2\|_{L^p(\Omega)}\,ds 
     \\
       &\leq \|A^{\vartheta}u_0\|_{L^2(\Omega)} 
                + C_{31}\int_0^{\infty} 
\sigma^{-\vartheta-\frac{3}{2}(\frac{1}{p}-\frac{1}{2})}
e^{-\lambda\sigma}\,d\sigma \leq C_{32}
      \end{align*}
for all $t \in (0, \tmax)$ 
since $\vartheta+\frac{3}{2}(\frac{1}{p}-\frac{1}{2}) < 1$. 
Moreover, the properties of $D(A^\vartheta)$ 
(\cite[Theorem 3]{Giga_1981-domain} 
and \cite[Theorem 1.6.1]{Henry_1981}) 
imply that there exists $C_{33}>0$ such that 
\begin{align*}
  \lp{\infty}{u(\cdot,t)}\le 
  C_{33}\lp{2}{A^\vartheta u(\cdot,t)}
\end{align*}
for all $t\in (0,\tmax)$, 
which concludes the proof.
\end{proof}

 
\begin{lem}\label{pote12}
Assume $\frac{\chi}{\mu} < \xi_0$. Then there exist 
$r \in (3, 6) \cap (1, q]$ and $C_{34}=C_{34}(r) > 0$ such that 
     $$
     \|\nabla c(\cdot, t)\|_{L^r(\Omega)} \leq C_{34} 
     \quad \mbox{for all}\ t \in (0, \tmax).
     $$
\end{lem}
\begin{proof} 
Let $r \in (3, 6)\cap(1, q]$ and fix $p \in (\frac{3q}{q+3}, 2)$.
An application of 
the variation of constants formula for $c$ 
leads to  
     \begin{align}\label{pote6.1}
     \|\nabla c(\cdot, t)\|_{L^{r}(\Omega)} 
     &\leq \|\nabla e^{t(\Delta - 1)} c_0\|_{L^{r}(\Omega)} 
         \\ \notag
              &\quad\,
              + \int_{0}^t 
                    \|\nabla e^{(t-s)(\Delta - 1)}
              (\alpha n_1(\cdot, s) + \beta n_2(\cdot, s)+1)c(\cdot, s)\|_{L^{r}(\Omega)}\,ds 
              \\ \notag
              &\quad\,+ \int_{0}^t 
                    \|\nabla e^{(t-s)(\Delta - 1)}
                                      \nabla\cdot(u(\cdot, s)c(\cdot, s))\|_{L^{r}(\Omega)}\,ds. 
     \end{align} 
We first obtain the estimate for the first term 
on the right-hand side of \eqref{pote6.1}. 
Noting that $q > 3$, we derive from the H\"older inequality and \cite[Lemma 1.3 (iii)]{Winkler_2010_Aggreagation} that 
there exist constants $C_{35},C_{36} > 0$ such that 
    \begin{align}\label{hangunc}
    \|\nabla e^{t(\Delta-1)}c_0\|_{L^r(\Omega)} 
    &\leq C_{35}\|\nabla e^{t(\Delta-1)}c_0\|_{L^q(\Omega)}
    \\\notag
    &\leq C_{36}\|\nabla c_0\|_{L^q(\Omega)}.
    \end{align}
We next establish the estimate for the second term 
on the right-hand side of \eqref{pote6.1}. 
Lemmas \ref{lem;Linf;c} and \ref{pote8} yield that 
there exist constants $C_{37}, C_{38} > 0$ such that 
       \begin{align}\label{intnorm}
       &\int_{0}^t 
                    \|\nabla e^{(t-s)(\Delta - 1)}
              (\alpha n_1(\cdot, s) + \beta n_2(\cdot, s)+1)c(\cdot, s)\|_{L^{r}(\Omega)}\,ds 
     \\ \notag
     &\leq C_{37}\int_{0}^t 
                             [1 + (t - s)^{-\frac{1}{2} - \frac{3}{2}(\frac{1}{p}-\frac{1}{r})}]
                 e^{-(t-s)}(\|n_1(\cdot, s)\|_{L^p(\Omega)} + \|n_2(\cdot, s)\|_{L^p(\Omega)}+|\Omega|^{\frac{1}{p}})\,ds 
        \\ \notag 
     &\leq C_{38}\int_{0}^t 
                 [1 + (t - s)^{-\frac{1}{2} - \frac{3}{2}(\frac{1}{p}-\frac{1}{r})}]e^{-(t-s)}\,ds.
     \end{align} 
Here, since $\frac{1}{2} + \frac{3}{2}(\frac{1}{p}-\frac{1}{r}) < 1$, we have 
      \begin{equation}\label{kaseki}
      \int_{0}^t [1 + (t - s)^{-\frac{1}{2} - \frac{3}{2}(\frac{1}{p}-\frac{1}{r})}]e^{-(t-s)}\,ds 
      \leq \int_0^{\infty}(1 + \sigma^{-\frac{1}{2} - \frac{3}{2}(\frac{1}{p}-\frac{1}{r})})e^{-\sigma}\,d\sigma < \infty.
      \end{equation}
Combination of \eqref{intnorm} and \eqref{kaseki} derives that  
    \begin{equation}\label{es}
    \int_{0}^t \|\nabla e^{(t-s)(\Delta - 1)}
                    (\alpha n_1(\cdot, s) + \beta n_2(\cdot, s)+1)c(\cdot, s)\|_{L^{r}(\Omega)}\,ds 
    \leq C_{39}
    \end{equation}
with some constant $C_{39}>0$.
Finally we will deal with the third term 
on the right-hand side of \eqref{pote6.1}. 
Now we put $\theta \geq r$, $0 < \ell < \frac{1}{2}$ 
satisfying $\frac{1}{2} + \frac{3}{2}(\frac{1}{\theta}-\frac{1}{r}) < \ell$ and 
$a \in (0, \frac{1}{2}-\ell)$. Then we derive from 
\cite[Section 2]{HW-2005} and Lemmas \ref{lem;Linf;c} and \ref{pote9} that 
there exist constants $C_{40}, C_{41}, C_{42} > 0$ 
and $\lambda>0$ such that 
     \begin{align}\label{hangun}
     &\int_{0}^t \|\nabla e^{(t-s)(\Delta - 1)}
                                      \nabla\cdot(u(\cdot, s)c(\cdot, s))\|_{L^{r}(\Omega)}\,ds 
     \\ \notag
     &\leq \int_{0}^t \|e^{(t-s)(\Delta - 1)}
                                  \nabla\cdot(u(\cdot, s)c(\cdot, s))\|_{W^{1, r}(\Omega)}\,ds 
     \\ \notag
     &\leq C_{40}\int_{0}^t \|(-\Delta+1)^{\ell}e^{(t-s)(\Delta - 1)}
                            \nabla\cdot(u(\cdot, s)c(\cdot, s))\|_{L^{\theta}(\Omega)}\,ds  
     \\ \notag
     &\leq C_{41}\int_{0}^t (t-s)^{-\ell-\frac{1}{2}-a}e^{-\lambda(t-s)}\|u(\cdot, s)c(\cdot, s)\|_{L^{\theta}(\Omega)}\,ds 
     \\ \notag
     &\leq C_{42}\int_{0}^t (t-s)^{-\ell-\frac{1}{2}-a}e^{-\lambda(t-s)}\,ds.
     \end{align}
Noting that $\ell+\frac{1}{2}+a < 1$, 
we infer that there exists a constant $C_{43}>0$ such that 
     \begin{equation}\label{kasekibun}
     \int_{0}^t (t-s)^{-\ell-\frac{1}{2}-a}e^{-\lambda(t-s)}\,ds 
     \leq \int_{0}^{\infty} \sigma^{-\ell-\frac{1}{2}-a}e^{-\lambda\sigma}\,d\sigma 
     \le C_{43}. 
     \end{equation}
From \eqref{hangun} and \eqref{kasekibun} we have
     \begin{equation}\label{kasekibunbun}
     \int_{0}^t \|\nabla e^{(t-s)(\Delta - 1)}
                                      \nabla\cdot(u(\cdot, s)c(\cdot, s))\|_{L^{r}(\Omega)}\,ds 
     \leq C_{44}
     \end{equation}
with some constant $C_{44} > 0$.
Therefore in light of \eqref{pote6.1}, 
\eqref{hangunc}, \eqref{es} and  \eqref{kasekibunbun} 
there exists a constant $C_{45} > 0$ that 
          $$
          \|\nabla c(\cdot, t)\|_{L^r(\Omega)} 
          \leq C_{45}
          $$
for all $t \in (0, \tmax)$.
\end{proof}
 Then we will derive the $L^\infty$-estimate for 
 $n_i$ by using the well-known semigroup estimates 
 (see \cite{B-B-T-W}). 
\begin{lem}\label{pote13}
Assume $\frac{\chi}{\mu} < \xi_0$. 
Then there exists a constant $C_{46} > 0$ such that 
     $$
     \|n_i(\cdot, t)\|_{L^{\infty}(\Omega)} \leq C_{46} 
     \quad \mbox{for all}\ t \in (0, \tmax)
     $$
for $i = 1, 2$.
\end{lem}
\begin{proof}
We let $q>3$ and let 
$p\in(\frac{3}{2},2)$ with $\frac{3p}{3-p} < q$. 
Then thanks to Lemma \ref{pote8}, we obtain 
\begin{align*}
  \lp{p}{n_1(\cdot,t)}\le C_{47} 
\end{align*}
for all $t\in(0,\tmax)$ with some $C_{47}>0$. 
Now we can choose $r\in (3,q)$ such that 
$p>\frac{3r}{3+r}$ 
and $\theta > 1$ satisfying 
\begin{align*}
\frac{1}{\theta} < \min\left\{1-\frac{r(3-p)}{3p},
\frac{q-r}{q}\right\}, 
\end{align*}
and put $\theta':=\frac{\theta}{\theta-1}$, and then 
\begin{align*} 
  r\theta'<\frac{3p}{3-p}\quad \mbox{and}\quad r\theta' < q
\end{align*} 
hold. 
Now for all $T'\in (0,\tmax)$ 
we note that 
\begin{align*}
M(T'):=\sup_{t\in (0,T')}\lp{\infty}{n_1(\cdot,t)} 
\end{align*}
is finite. 
In order to obtain the estimate for $M(T')$ 
for all $t\in (0,T')$ we put $t_0:=(t-1)_+$ and 
represent $n_1$ according to 
\begin{align*}
  n_1(\cdot,t)
  &=e^{(t-t_0)\Delta}n_1(t_0) 
  - \int_{t_0}^t e^{(t-s)\Delta}\nabla\cdot 
  \left(\chi n_1(\cdot,s)\nabla c(\cdot,s) + n_1(\cdot,s)u(\cdot,s) \right)ds 
\\
&\quad\,+\mu_1\int_{t_0}^t e^{(t-s)\Delta} n_1(1- n_1 - a_1n_2)ds
  \\
  &=:I_1(\cdot,t)+I_2(\cdot,t)+I_3(\cdot,t).
\end{align*}
In the case that $t\le 1$, 
from the order preserving property of 
the Neumann heat semigroup we know that 
\begin{align*}
  \lp{\infty}{I_1(\cdot,t)}
  \le \lp{\infty}{n_{1,0}}
  \quad 
  \mbox{for all}\ t\in(0,\min\{1,T'\}). 
\end{align*}
In the case that $t>1$, 
by using the $L^p$-$L^q$ estimate for 
$(e^{\tau \Delta})_{\tau> 0}$ 
(see \cite[Lemma 1.3 (i)]{Winkler_2010_Aggreagation}) 
and Lemma \ref{pote8} we can see that 
there exists a constant $C_{48}>0$ such that 
\begin{align*}
  \lp{\infty}{I_1(\cdot,t)}
  \le \lp{p}{n_1(\cdot,t_0)} 
  \le C_{48}
  \quad 
  \mbox{for all}\ t\in (1,T'). 
\end{align*}
Thanks to the elementary inequality 
    $$
    \mu_1n_1(1 - n_1 - a_1n_2) 
    \leq -\mu_1\left(n_1-\frac{1+\mu_1}{2\mu_1}\right)^2 
                                                               + \frac{(1+\mu_1)^2}{4\mu_1} 
    \leq \frac{(1+\mu_1)^2}{4\mu_1}
    $$
together with the maximum principle, 
we see that 
there exists a constant $C_{49}>0$ such that 
\begin{align*}
I_3(\cdot,t)\le C_{49}\quad 
  \mbox{for all}\ t\in (1,T'). 
\end{align*}
Next we obtain from the known smoothing property of 
$(e^{\tau\Delta})_{\tau\ge 0}$ (see \cite{FIWY_2016}) 
that 
\begin{align*}
  &\int_{t_0}^t\left\|e^{(t-s)\Delta}\nabla\cdot 
  \left(\chi_1 n_1(\cdot,s)\nabla c(\cdot,s) 
  + n_1(\cdot,s)u(\cdot,s) 
  \right)\right\|_{L^{\infty}(\Omega)}ds
\\
  &\le 
  C_{50} \sup_{s\in (0,T')}
  \left(\chi_1\left\|n_1(\cdot,s)\nabla c(\cdot,s)\right\|_{L^{\infty}(\Omega)} 
  + \left\|n_1(\cdot,s)u(\cdot,s)\right\|_{L^{\infty}(\Omega)}
  \right) 
  \int_0^1  \sigma^{-\frac{1}{2}-\frac{3}{2r}}\,d\sigma
\end{align*} 
for all $t\in(0,T')$ with some $C_{50}>0$. 
Here we note from $\frac{1}{3}+\frac{3}{2r}<1$ that 
$\int_0^1 \sigma^{-\frac{1}{3}-\frac{3}{2r}}\,d\sigma$ 
is finite. 
Then we can obtain that 
\begin{align*}
 \left\|n_1(\cdot,s)\nabla 
 c(\cdot,s)\right\|_{L^{\infty}(\Omega)}
 & \le 
 \left\|n_1(\cdot,s)\nabla 
 c(\cdot,s)\right\|_{L^{\infty}(\Omega)}
 \\
 & \le 
 \lp{r\theta}{n_1(\cdot,s)}\lp{r\theta'}{\nabla c(\cdot,s)}
 \\
 & \le 
 M(T')^{1-\frac{1}{r\theta}}
 \lp{1}{n_1(\cdot,s)}^{\frac{1}{r\theta}}
 \lp{r\theta'}{\nabla c(\cdot,s)} 
\end{align*}
for all $s\in (0,T')$. 
Noting from $r\theta'< \frac{3p}{3-p}$ that 
\begin{align*} 
r\theta' \in (3,6)\cap (1,q], 
\end{align*} 
we have from 
Lemma \ref{pote12} that 
there exists $C_{51}>0$ such that 
\begin{align*}
  \lp{r\theta'}{\nabla c(\cdot,s)}
  \le C_{51} 
  \quad\mbox{for all}\ s\in (0,T'). 
\end{align*}
Therefore we can find $C_{52}>0$ 
satisfying 
\begin{align*}
  \lp{\infty}{n_1(\cdot,s)\nabla c(\cdot,s)}
  \le C_{52}
  \quad \mbox{for all}\ s\in (0,T'). 
\end{align*}
Similarly, from Lemma \ref{pote9} 
there exists a constant $C_{53}>0$ such that 
\begin{align*}
 \lp{\theta}{n_1(\cdot,s)u(\cdot,s)}
 \le C_{53}M(T')^{1-\frac{1}{r\theta}}
 \lp{1}{n_1(\cdot,s)}.
\end{align*}
Therefore, Lemma \ref{pote1} leads to 
the existence of $C_{54},C_{55}>0$ such that 
\begin{align*}
  n_1(\cdot,t)
  \le \lp{\infty}{I_1(\cdot,t)}+\lp{\infty}{I_2(\cdot,t)}+I_3(\cdot,t)
  \le C_{54} + C_{55}M(T')^{1-\frac{1}{r\theta}},
\end{align*}
which implies from the positivity of $n_1$ that 
    $$
    M(T') \leq C_{54} + C_{55}M(T')^{1-\frac{1}{r\theta}}. 
    $$
Noting that $r\theta>1$, 
we derive that there exists $C_{56}>0$ such that 
\begin{align*}
 M(T')=\sup_{t\in(0,T')}
 \lp{\infty}{n_1(\cdot,t)}\le C_{56} 
 \quad 
 \mbox{for all}\ T'\in (0,\tmax).
\end{align*}
Similarly we prove that there exists 
a constant $C_{57}>0$ such that 
$\|n_2(\cdot, t)\|_{L^{\infty}(\Omega)} \leq C_{57}$ for all $t \in (0, \tmax)$.
Therefore we can attain the conclusion of the proof. 
\end{proof}


\begin{prth1.1}
Combination of Lemmas \ref{pre1}, \ref{pote9}, \ref{pote12} and \ref{pote13} directly 
leads to Theorem \ref{maintheorem1}. 
\qed
\end{prth1.1}

\section{Stabilization. Proof of Theorem \ref{maintheorem3}}

\subsection{Case 1: $a_1, a_2 \in (0, 1)$} 

Now we assume that $\frac{\chi}{\mu} < \xi_0$. 
In this section we will show stabilization in \eqref{P} 
in the case $a_1, a_2 \in (0, 1)$. 
We will prove the key estimate for the proof of Theorem \ref{maintheorem3}. 
The proof is same as that of \cite[Lemma 4.1]{HKMY_1}.

\begin{lem}\label{lem;energy;case1}
Let $a_1, a_2 \in (0, 1)$ 
and let $(n_1, n_2, c, u)$ be a solution to 
\eqref{P}. 
Under the assumption of Theorem \ref{maintheorem1}, 
there exist $k_1,\ell_1>0$ and $\ep_1>0$ 
such that the nonnegative functions $E_1$ and $F_1$ 
defined by 
\begin{align*}
  E_1:=\int_\Omega 
       \left(
         n_1-N_1\log \frac{n_1}{N_1}
       \right)
       + 
       k_1\int_\Omega 
       \left(
         n_2-N_2\log \frac{n_2}{N_2}
       \right)
       + \frac{\ell_1}{2}
       \int_\Omega c^2
\end{align*}
and 
\begin{align*}
F_1:=\int_\Omega (n_1-N_1)^2
     +
     \int_\Omega (n_2-N_2)^2
\end{align*}
satisfy 
\begin{align}\label{esti;energy}
  \frac d{dt}E_1(t) 
  \leq  
  - \ep_1 F_1(t)
\quad
  \mbox{for all}\ t>0, 
\end{align}
where 
       $$
       N_1 := \frac{1 - a_1}{1 - a_1a_2},\quad N_2 := \frac{1 - a_2}{1 - a_1a_2}.
       $$
\end{lem}

By using Lemma \ref{lem;energy;case1} we can show 
stabilization of $n_1,n_2$. 

\begin{lem}\label{n_1n_2che1}
Let $a_1, a_2 \in (0, 1)$ 
and let $(n_1, n_2, c, u)$ be a solution to \eqref{P}. 
Under the assumption of Theorem \ref{maintheorem1}, 
the solution of \eqref{P} 
satisfies the following properties\/{\rm :}  
     \begin{equation*}
     \|n_1(\cdot, t) - N_1\|_{L^{\infty}(\Omega)} \to 0,\quad 
     \|n_2(\cdot, t) - N_2\|_{L^{\infty}(\Omega)} \to 0\quad
     \mbox{as}\ t \to \infty.
     \end{equation*}
\end{lem}
\begin{proof}
Firstly we can see from Lemmas \ref{pote9}, \ref{pote12}, \ref{pote13} and \cite{Ladyzenskaja} that 
there exist constants $C_{58} > 0$ 
and $\alpha_0 \in (0, 1)$ 
such that  
      $$
      \|n_1\|_{C^{\alpha_0, \frac{\alpha_0}{2}}(\overline{\Omega}\times[t, t+1])} 
      +  \|n_2\|_{C^{\alpha_0, \frac{\alpha_0}{2}}(\overline{\Omega}\times[t, t+1])} 
      \leq C_{58}
      $$
for all $t \geq 1$. 
Now we set 
\[
  f_1(t) := \int_{\Omega} 
  (n_1 - N_1)^2 + \int_{\Omega} (n_2 - N_2)^2 \geq 0. 
\]
Then the function $f_1$ is nonnegative 
and uniformly continuous. 
We see from Lemma \ref{lem;energy;case1} that 
\begin{align}\label{f_1esti}
  \int_1^{\infty}\int_\Omega (n_1-N_1)^2 + 
  \int_1^\infty \int_\Omega (n_2-N_2)^2=
  \int_1^{\infty} f_1(t)\,dt \leq \frac{1}{\ep_1}E_1(1) < \infty.
\end{align}
Thus the compactness method (\cite[Lemma 4.6]{HKMY_1}) 
concludes the proof. %
\end{proof}


\subsection{Case 2: $a_1 \geq 1 > a_2$} 

In this section 
we assume that $\frac{\chi}{\mu} < \xi_0$. 
This section is devoted to obtaining 
stabilization in \eqref{P} 
in the case $a_1 \geq 1 > a_2$. 
We will give the following lemma for obtaining it. 
The proof is same as that of 
\cite[Lemma 4.3]{HKMY_1}. 

\begin{lem}\label{lem;energy;case2}
Let $a_1 \geq 1> a_2$  
and let $(n_1, n_2, c, u)$ be a solution to { \eqref{P}}. 
Under the assumption of Theorem \ref{maintheorem1}, 
there exist $k_2,\ell_2>0$ and $\ep_2>0$ 
such that the nonnegative functions $E_2$ and $F_2$ 
defined by 
\begin{align*}
  E_2:=\int_\Omega n_1
       + 
       k_2\int_\Omega 
       \left(
         n_2-\log n_2
       \right)
       + \frac{\ell_2}{2}
       \int_\Omega c^2
\end{align*}
and 
\begin{align*}
F_2:=\int_\Omega n_1^2
     +
     \int_\Omega (n_2-1)^2
\end{align*}
satisfy 
\begin{align}\label{esti;energy;case2}
  \frac d{dt}E_2(t) 
  \leq  
  - \ep_2 F_2(t)
\quad
  \mbox{for all}\ t>0. 
\end{align}
\end{lem}


By using a similar argument, 
Lemma \ref{lem;energy;case2} 
leads to 
stabilization of $n_1,n_2$. 
\begin{lem}\label{n_1n_2che2}
Let $a_1 \geq 1> a_2$ 
and let $(n_1, n_2, c, u)$ be a solution to \eqref{P}. 
Under the assumption of Theorem \ref{maintheorem1}, it holds that 
     $$
     \|n_1(\cdot, t)\|_{L^{\infty}(\Omega)} \to 0,\quad 
     \|n_2(\cdot, t) - 1\|_{L^{\infty}(\Omega)} \to 0\quad
     \mbox{as}\ t \to \infty.
     $$
\end{lem}
\begin{proof}
Noting from Lemma \ref{esti;energy;case2} that 
\begin{align}\label{f_2esti}
  \int_1^\infty \int_\Omega n_1^2 
  +\int_1^\infty \int_\Omega (n_2-1)^2 <\infty, 
\end{align}
we can prove this lemma by the same argument as in 
the proof of Lemma \ref{n_1n_2che1}.
\end{proof}

\subsection{Convergence for $c$ and $u$}

Finally we give the following lemma to establish 
the decay properties of $c$ and $u$.  
We first show the lower estimate for $n_2$. 
\begin{lem}\label{esti;low;n}
Let $a_2\in (0,1)$. 
Under the assumption of 
Theorem \ref{maintheorem1}, 
there exist constants $C_{59}>0$ and $T^{*}>0$ such that 
\begin{align*}
  n_2(x, t) \geq C_{59}
  \quad 
  \mbox{for all}\
  x \in \Omega\ \mbox{and all}\ t>T^{*}.
\end{align*}
\end{lem}
\begin{proof} 
We first deal with the case 
that $a_1,a_2\in (0,1)$. 
Now, we assume that this lemma does not hold. 
Then there exist $\{x_j\}_{j\in\mathbb{N}} \subset \Omega$ and 
$\{t_j\}_{j\in\mathbb{N}} \subset [0, \infty)$ 
such that $t_j\to\infty$ as $j\to\infty$ satisfying  
   \[
   n_2(x_j, t_j) < \frac{N_2}{2} 
   \quad \mbox{for all}\ j\in \mathbb{N}.
   \]
Thus we have 
    \begin{align}\label{notconvergence}
    \|n_2(\cdot, t_j) - N_2\|_{L^{\infty}(\Omega)} 
    \geq N_2 - n_2(x_j, t_j) > \frac{N_2}{2} \quad \mbox{for all}\ j\in\mathbb{N}, 
    \end{align}
which means that $n_2(\cdot, t_j)$ does not converge to $N_2$ as $j\to\infty$. 
However, 
Lemma \ref{n_1n_2che1} asserts that 
\[
  \lp{\infty}{n_2(\cdot,t_j)-N_2}\to 0 
  \quad \mbox{as}\ j\to \infty,
\]
which contradicts 
\eqref{notconvergence}. 
In the case that $a_1\ge 1>a_2$ a similar argument 
leads to the lower estimate for $n_2$. 
Therefore we can conclude the proof. 
\end{proof}

\begin{lem}\label{lem;decaycu}
Under the assumption of 
Theorem \ref{maintheorem1}, 
the solution of \eqref{P} satisfies 
  \begin{align*}
    \lp{\infty}{c(\cdot,t)}\to 0,
    \quad 
    \lp{\infty}{u(\cdot,t)}\to 0
    \quad \text{as}\ t\to\infty.
  \end{align*}
\end{lem}
\begin{proof}
Noting from Lemmas \ref{lem;energy;case1} and \ref{lem;energy;case2} 
that 
$
 \int_{T^{*}}^\infty \int_\Omega (\alpha n_1+\beta n_2)c^2 <\infty 
$
and 
using Lemma \ref{esti;low;n}, 
we can establish that 
\[
\int_{T^{*}}^\infty\int_\Omega c^2<\infty,
\] 
which entails
 $\lp{\infty}{c(\cdot,t)}\to 0$ as $t\to\infty$. 
Next we will show that $\lp{\infty}{u(\cdot,t)}\to 0$ as $t\to\infty$.
Let $\theta \in (\frac{3}{4}, \vartheta)$ 
and $a=\frac{\theta}{\vartheta}\in (0,1)$. 
It follows from combination of 
\cite[Theorem 3]{Giga_1981-domain}, 
\cite[Theorem 1.6.1]{Henry_1981}, 
\cite[Part 2, Theorem 14.1]{Friedman_partial} 
and Lemma \ref{pote9} that 
there exist constants $C_{60}, C_{61}, C_{62} > 0$ such that 
    \begin{equation*}
    \lp{\infty}{u} \leq C_{60}\lp{2}{A^{\theta}u} 
    \leq C_{61}\lp{2}{A^{\vartheta}u}^{a}\lp{2}{A^0u}^{1-a} 
    \leq C_{62}\lp{2}{u}^{1-a}, 
    \end{equation*} 
which means that it is sufficient to show that 
      $$
      \lp{2}{u(\cdot,t)}\to 0
    \quad \text{as}\ t\to\infty.
      $$
We first note from the Poincar\'e inequality 
that there exists a constant $C_{63} > 0$ such that 
    \begin{equation*}
    \|u(\cdot, s)\|_{L^2(\Omega)}^2 \leq C_{63}\|\nabla u(\cdot, s)\|_{L^2(\Omega)}^2
    \end{equation*}
for all $s \in (0, \infty)$. 
Put 
$(n_{1, \infty},n_{2, \infty}):= (N_1,N_2)$ 
if $a_1,a_2\in (0,1)$ or 
$(n_{1, \infty},n_{2, \infty}):= (0,1)$ 
if $a_1\ge 1 >a_2$. 
We infer from the fourth equation in \eqref{P} 
and the Young inequality that 
    \begin{align*}
    &\frac{1}{2}\frac{d}{dt}\int_{\Omega} |u|^2 + \int_{\Omega} |\nabla u|^2 
    \\
    &= \int_{\Omega} (\gamma(n_1 - n_{1, \infty}) + \delta(n_2 - n_{2, \infty}))\nabla\phi\cdot u + (\gamma n_{1, \infty} + \delta n_{2, \infty})\int_{\Omega}\nabla\phi\cdot u
\\
    &\leq \frac{1}{4C_{63}}\int_{\Omega} |u|^2 
             + C_{64}\Bigl(\int_{\Omega} (n_1 - n_{1, \infty})^2 + \int_{\Omega} (n_2 - n_{2, \infty})^2\Bigr) 
     + (\gamma n_{1, \infty} + \delta n_{2, \infty})\int_{\Omega}\nabla\phi\cdot u
    \end{align*} 
for all $t \in (0, \infty)$ with some constant 
$C_{64} > 0$. Since $\nabla \cdot u = 0$ in $\Omega \times (0, \infty)$, the functions 
$$
y(t) := \int_{\Omega}|u(\cdot, t)|^2 \quad\mbox{and}\quad 
h(t) := 2C_{64}\Bigl(\int_{\Omega} (n_1 - n_{1, \infty})^2 + \int_{\Omega} (n_2 - n_{2, \infty})^2\Bigr)
$$
satisfy  
     \begin{equation*}
     y'(t) + C_{65}y(t) \leq h(t)
     \end{equation*}
with some $C_{65} > 0$.
Hence it holds that 
     \begin{align}\label{pepe}
     y(t) 
     &\leq y(0)e^{-C_{65}t} 
     + \int_{0}^{t}e^{-C_{65}(t-s)}h(s)\,ds  
\\ \notag
     &\leq y(0)e^{-C_{65}t} 
     + \int_{0}^{\frac{t}{2}}e^{-C_{65}(t-s)}h(s)\,ds 
              + \int_{\frac{t}{2}}^{t}e^{-C_{65}(t-s)}h(s)\,ds. 
     \end{align} 
Here we see from Lemma \ref{pote13} that 
there exists a constant $C_{66}>0$ 
such that $h(s) \leq C_{66}$ for all $s > 0$, and hence 
we have 
     \begin{equation}\label{pen}
     \int_{0}^{\frac{t}{2}}e^{-C_{65}(t-s)}h(s)\,ds 
     \leq C_{66}e^{-C_{65}t}\int_{0}^{\frac{t}{2}}
     e^{C_{65}s}\,ds 
     \leq C_{67}e^{-\frac{C_{65}}{2}t}
     \end{equation}
with some $C_{67} > 0$. 
On the other hand, noting from 
\eqref{f_1esti} and \eqref{f_2esti} that 
$\int_{0}^{\infty} h(s)\,ds < \infty$, 
we can see that 
     \begin{equation}\label{pika}
      0 \le 
      \int_{\frac{t}{2}}^{t}e^{-C_{65}(t-s)}h(s)\,ds 
      \le 
      \int_{\frac{t}{2}}^{t}h(s)\,ds \to 0 
      \quad \mbox{as}\ t\to\infty.
     \end{equation} 
Therefore combination of \eqref{pepe} 
with \eqref{pen} and \eqref{pika} leads to 
      \begin{equation*}
      \lp{2}{u(\cdot, t)}^2 = y(t) \to 0  
      \quad \mbox{as}\ t\to\infty, 
     \end{equation*} 
which means the end of the proof.
\end{proof}

\subsection{Proof of Theorem \ref{maintheorem3}}

\begin{prth1.2}
Lemmas \ref{n_1n_2che1}, \ref{n_1n_2che2}, 
and \ref{lem;decaycu} directly show 
Theorem \ref{maintheorem3}. 
\qed
\end{prth1.2}

 
\end{document}